\documentclass[10pt]{article}
\usepackage{amsmath,amssymb,wasysym}
\usepackage{enumitem}
\usepackage[nospace]{cite}

\newcommand{\N}{\mathbb{N}}
\newcommand{\Z}{\mathbb{Z}}
\newcommand{\R}{\mathbb{R}}
\newcommand{\C}{\mathbb{C}}

\newcommand{\ep}{\epsilon}

\newcommand{\ffi}{\varphi}
\newcommand{\e}{\mathrm{e}}

\newcommand{\wb}{\overline}
\newcommand{\diff}{\,\mathrm{d}}
\newcommand{\dif}{\mathrm{d}}

\let\le=\leqslant
\let\ge= \geqslant

\newcommand{\norm}[1]{\ensuremath{\left\Vert #1 \right\Vert }}

\newtheorem{theorem}{Theorem}
\newtheorem{remark}{Remark}

\newenvironment{proof}
{\hspace{-3mm}\textit{Proof.}}{\nolinebreak\hfill$\Box$}

\usepackage[margin={3.8cm,3.3cm}]{geometry}

\numberwithin{equation}{section}

\begin{document}

\title{\Large Instability of an integrable nonlocal NLS}

\author{Fran\c cois Genoud \\
\small\emph{Delft Institute of Applied Mathematics, Delft University of Technology}\\
\small\emph{Mekelweg~4, 2628 CD Delft, The Netherlands}\\
\small s.f.genoud@tudelft.nl}

\date{}



\maketitle
 
\begin{abstract}
In this note we discuss the global dynamics of an integrable nonlocal NLS
on $\R$, which has been the object of recent investigation by integrable systems methods.
We prove two results which are in striking contrast with the case of the local cubic focusing NLS on $\R$. 
First, finite time blow-up solutions exist with arbitrarily small initial data in $H^s(\R)$, for any $s\ge0$.
On the other hand, the solitons of the local NLS, which are also solutions of the nonlocal equation, are unstable
by blow-up for the latter.
\end{abstract}


\section{Introduction}

The nonlocal nonlinear Schr\"{o}dinger equation
\begin{equation}\label{nls}
iu_t(t,x)+u_{xx}(t,x)+u^2(t,x)\wb{u}(t,-x)=0, \qquad u(t,x):\R\times\R \to \C,
\end{equation}
has recently been shown to be a completely integrable system, with infinitely many conservation laws \cite{mus1,mus2}.
The equation is related to two different areas of physics: gain/loss systems in optics and so-called
$PT$-symmetric quantum mechanics, see \cite{mus0,yang,ben,mos} and references therein. 
Mathematically, the feature connecting \eqref{nls} to these areas
is the $PT$-symmetry of the `nonlinear potential' $u(t,x)\wb{u}(t,-x)$. Namely,
this quantity is invariant under the joint transformation $x\to -x$ and $i\to -i$ (parity and time reversal).

The inverse-scattering transform was applied in \cite{mus1,mus2} to produce a variety of solutions to \eqref{nls}. 
In particular, a `one-soliton solution' is obtained, which blows up in finite time (actually, up to rescaling, 
at all times $t=2m+1, \ m\in\Z$). The purpose of this note is to use this peculiar solution to prove some results
about the global dynamics of \eqref{nls}, which are in striking contrast with the case of
the local focusing cubic equation
\begin{equation}\label{local_nls}
iu_t(t,x)+u_{xx}(t,x)+|u(t,x)|^2u(t,x)=0, \qquad u(t,x):\R\times\R \to \C.
\end{equation}

We will first show that \eqref{nls} is locally well-posed (in $H^1(\R)$) but  
then we prove that there exist solutions which blow up in finite time (in $L^\infty(\R)$), 
with arbitrarily small initial data in $H^s(\R)$, for any $s\ge0$. This
shows in particular that the trivial solution, $u\equiv0$, is unstable by blow-up.

Let us now observe that \eqref{nls} reduces to \eqref{local_nls}
provided the discussion is restricted to even solutions. 
Hence, the well-known solitons 
$u_\omega(t,x)=\e^{i\omega t}\ffi_\omega(x)$ of \eqref{local_nls}, 
where
\begin{equation}\label{sech}
\ffi_\omega(x)=\frac{2\sqrt{2\omega}}{\e^{\sqrt{w}x}+\e^{-\sqrt{w}x}}, \qquad \omega>0,
\end{equation}
are also solutions of \eqref{nls}. These standing waves are orbitally stable with respect to \eqref{local_nls} but
we show that they are unstable by blow-up with respect to \eqref{nls}.  

The rest of the paper is organised as follows. In Section~\ref{well} we prove the local well-posedness
and the blow-up instability of the zero solution. In Section~\ref{blow} we prove the blow-up instability of the solitons 
\eqref{sech}. We conclude in Section~\ref{defoc} with some remarks on the `defocusing' equation 
\begin{equation}\label{defoc_nls}
iu_t(t,x)+u_{xx}(t,x)-u^2(t,x)\wb{u}(t,-x)=0, \qquad u(t,x):\R\times\R \to \C.
\end{equation}

\subsection*{Notation} For non-negative quantities $A,B$ we write $A \lesssim B$ if $A\le CB$ for some
constant $C>0$, whose exact value is not essential to the analysis.

\subsection*{Acknowledgement} I am grateful to Stefan Le Coz for an interesting discussion about 
\eqref{nls}, and to the anonymous referee, whose remarks helped improve the presentation of the paper.


\section{Instability of the trivial solution}\label{well}

We start with a local well-posedness result.

\begin{theorem}\label{local.thm}
Given any initial data $u_0\in H^1(\R)$, there exists a unique maximal solution 
$u\in C\big([0,T_\mathrm{max}),H^1(\R)\big)$ of \eqref{nls} such that $u(0,\cdot)=u_0$, where 
$T_\mathrm{max}=T_\mathrm{max}(\norm{u_0}_{H^1(\R)})$.
\end{theorem}

\begin{proof}
The theorem is proved by a fixed point argument, similar to the
case of the local equation \eqref{local_nls}. However, some calculations are different due to the 
nonlocal nonlinearity, so we give the proof here for completeness.

Fix $u_0\in H^1(\R)$, define $F(u)(x)=u^2(x)\wb{u}(-x)$ and a map 
$\tau: X_T\to X_T$ by
\[
\tau(u)(t)=S(t)u_0+i\int_0^tS(t-s)F(u)(s) \diff s,
\]
where $X_T=L^\infty\big((0,T);H^1(\R)\big)$ for some $T>0$ and $S(t)$ is the free Schr\"odinger group.
We shall prove the existence of a unique fixed point of $\tau$ in the ball 
\[
B_R=\{u\in X_T : \norm{u}_{X_T}<R\},
\]
for suitable values of $T,R>0$. That this fixed point can be extended to a maximal solution
$u\in C\big([0,T_\mathrm{max}),H^1(\R)\big)$ of \eqref{nls} then follows by standard arguments.

First observe that, for any $p\ge2$, the Sobolev embedding theorem yields
\begin{align*}
\norm{F(u)}_{L^p}^p	&= \int_\R |u(x)|^{2p}|u(-x)|^p \diff x \\
					&\le \Big\{\int_\R |u(x)|^{2pr}\diff x\Big\}^{1/r}\Big\{\int_\R |u(-x)|^{ps}\diff x\Big\}^{1/s} \\
					&= \norm{u}_{L^{2pr}}^{2p}\norm{u}_{L^{ps}}^{p}
					\lesssim \norm{u}_{H^1}^{3p},	
\end{align*}
where $r,s\ge1$ are arbitrary H\"older conjugate exponents. It follows that
\begin{equation}\label{Festimate}
\norm{F(u)}_{L^p} \lesssim \norm{u}_{H^1}^{3} \quad \ \text{for any} \ p\ge 2,
\end{equation}
and a similar estimate yields
\begin{equation}\label{Fxestimate}
\norm{F(u)_x}_{L^p} \lesssim \norm{u}_{H^1}^{3} \quad \text{for any} \ p\ge 2,
\end{equation}
where 
\begin{equation}\label{F_x}
F(u)_x=\big[u^2(x)\wb{u}(-x)\big]_x=2u(x)u_x(x)\wb{u}(-x)-u^2(x)\wb{u}_x(-x).
\end{equation}
By Strichartz's estimate and \eqref{Festimate}--\eqref{Fxestimate} with $p=2$, 
we see in particular that $\tau$ indeed maps $X_T$ into $X_T$. Furthermore,
there exist constants $C_1,C_2>0$ such that
\begin{align*}\label{tauestimate}
\norm{\tau(u)}_{X_T}	&\le C_1\norm{u_0}_{H^1}+T\norm{F(u)}_{L^\infty(0,T;H^1)} \\
				&\le C_1\norm{u_0}_{H^1}+TC_2\norm{u}^3_{X_T} .
\end{align*}
Choosing $R=2C_1\norm{u_0}_{H^1}$ and $T>0$ such that $C_1TR^2=1/2$, it follows that, for any
$u\in B_R$,
\[
\norm{\tau(u)}_{X_T} \le \frac{R}{2}+TC_1 \norm{u}^2_{X_T}\norm{u}_{X_T}
				\le \frac{R}{2}+\frac12 \norm{u}_{X_T}<R.
\]
Hence, for these values of $T,R>0$, $\tau$ maps the ball $B_R$ into itself. 

We now show that, if $T>0$ is small enough, then $\tau$ is a contraction in $B_R$. We have
\begin{equation}\label{taudiff}
\norm{\tau(u)-\tau(v)}_{X_T}\le T\norm{F(u)-F(v)}_{L^\infty(0,T;H^1)}, \quad u,v\in X_T.
\end{equation}
Writing
$|F(u)-F(v)|=|\int_0^1 \frac{\dif}{\dif\theta} F(\theta u +(1-\theta) v) \diff\theta|$,
we obtain
\[
|F(u)-F(v)|\lesssim |u(x)+v(x)||u(x)-v(x)||u(-x)+v(-x)|+|u(x)+v(x)|^2|u(-x)-v(-x)|
\]
and it follows that
\begin{equation}\label{FL2}
\norm{F(u)-F(v)}_{L^2}\lesssim \big(\Vert{u}\Vert_{H^1}^2+\Vert{v}\Vert_{H^1}^2\big)\norm{u-v}_{L^2}.
\end{equation}
On the other hand, in view of \eqref{F_x}, letting 
\[
G(u)(x)=2u(x)u_x(x)\wb{u}(-x) \quad\text{and}\quad H(u)(x)=u^2(x)\wb{u}_x(-x)
\]
we have
\[
|[F(u)-F(v)]_x| \le |G(u)-G(v)|+|H(u)-H(v)|,
\]
where
\begin{align*}
|G(u)-G(v)| 	&\lesssim |u(x)-v(x)||u_x(x)+v_x(x)||u(-x)+v(-x)| \\ 
			&\phantom{\lesssim}+|u(x)+v(x)||u_x(x)-v_x(x)||u(-x)+v(-x)| \\
			&\phantom{\lesssim}+|u(x)+v(x)||u_x(x)+v_x(x)||u(-x)-v(-x)|
\end{align*}
and
\begin{align*}
|H(u)-H(v)| 	&\lesssim |u(x)+v(x)||u(x)-v(x)||u_x(-x)+v_x(-x)| \\ 
			&\phantom{\lesssim}+|u(x)+v(x)|^2|u_x(-x)+v_x(-x)|.
\end{align*}
It follows that
\begin{equation}\label{F_xL2}
\norm{[F(u)-F(v)]_x}_{L^2}\lesssim  \big(\Vert{u}\Vert_{H^1}^2+\Vert{v}\Vert_{H^1}^2\big)\norm{u-v}_{H^1}.
\end{equation}
By \eqref{taudiff}, \eqref{FL2} and \eqref{F_xL2}, there is a constant $C>0$ such that
\[
\norm{\tau(u)-\tau(v)}_{X_T}\le CT\big(\Vert{u}\Vert_{X_T}^2+\Vert{v}\Vert_{X_T}^2\big)\norm{u-v}_{X_T}.
\]
Hence, if $u,v\in B_R$ we have
\[
\norm{\tau(u)-\tau(v)}_{X_T}\le 2CTR^2\norm{u-v}_{X_T},
\]
showing that $\tau$ is a contraction in $B_R$ provided $T<(2CR^2)^{-1}$.
The contraction mapping principle now yields a unique fixed point of $\tau$ in $B_R$,
which concludes the proof.
\end{proof}

\medskip
For the local equation \eqref{local_nls}, the next chapter of the story is well known. 
One proves that, for any $u_0\in H^1(\R)$, the maximal solution is global, i.e.~that 
$T_\mathrm{max}=\infty$. This is usually done by means of the energy and charge functionals
\[
E(u)=\frac12\int_\R |u_x|^2\diff x -\frac14\int_\R|u|^4 \diff x, \qquad
Q(u)=\frac12\int_\R |u|^2\diff x.
\]
Using the conservation of these quantities along the flow and
the Gagliardo--Nirenberg inequality, one shows that
the first term in $E$ is controlled by the second one, and must remain bounded. Hence, global existence in
$H^1(\R)$ is ensured by the blow-up alternative.

The corresponding conservation laws for \eqref{nls} are \cite{mus1,mus2}
\[
E(u)=\frac12\int_\R u_x(x)\wb{u}_x(-x)\diff x -\frac14\int_\R u^2(x)\wb{u}^2(-x) \diff x \quad\text{and}\quad
Q(u)=\frac12\int_\R u(x)\wb{u}(-x)\diff x.
\]
Even though each of these integrals is real, in general none of the three terms has a definite sign, unless
$u$ is even (or odd), in which case we recover the energy and charge of the local equation \eqref{local_nls}.
This predicament wipes away any hope of proving a global well-posedness result for \eqref{nls},
even for small initial data. In fact, we have the following result.

\begin{theorem}\label{global.thm}
For any $0<\alpha<1$, there exists a solution $u^{\alpha}(t,x)$ of \eqref{nls}, defined on $[0,T_\alpha)\times\R$,
where $T_\alpha=\pi/3\alpha^2$, with the following properties:
\begin{itemize}
\item[(i)] $u^{\alpha}$ blows up in $L^\infty(\R)$ as $t\to T_\alpha$, with
$\lim_{t\to T_\alpha}|u^{\alpha}(t,0)|=\infty$;
\item[(ii)] $u^{\alpha}_0=u^{\alpha}(0,\cdot)$ satisfies
$\norm{u^{\alpha}_0}_{H^k(\R)} \lesssim \alpha^{1/2}$, for all $k\in \N$.
\end{itemize}
\end{theorem}

\begin{proof}
The result is obtained from the explicit solution
\begin{equation}\label{2param}
u^{\alpha,\beta}(t,x)=\frac{2\sqrt{2}(\alpha+\beta)}{\e^{-4i\alpha^2t}\e^{2\alpha x}+\e^{-4i\beta^2t}\e^{-2\beta x}}.
\end{equation}
For any $\alpha,\beta>0, \alpha\neq\beta$, this function blows up at all times
\[
T_m=\frac{(2m+1)\pi}{4(\alpha^2-\beta^2)}, \qquad m\in \Z,
\]
with $\lim_{t\to T_m}|u^{\alpha,\beta}(t,0)|=\infty$, and is a solution of \eqref{nls}
in the sense of Theorem~\ref{local.thm} between these times, 
i.e.~$u^{\alpha,\beta}\in C\big((T_m,T_{m+1}),H^1(\R)\big), \ m\in \Z$.
To simplify the analysis, we choose $\beta=\alpha/2$, so that $u^{\alpha,\beta}$ reduces to
\begin{equation}\label{1param}
u^{\alpha}(t,x)=\frac{3\sqrt{2}\alpha}{\e^{-4i\alpha^2t}\e^{2\alpha x}+\e^{-4i\beta^2t}\e^{-\alpha x}},
\end{equation}
and the first blow-up time to the right of $t=0$ becomes
\[
T_\alpha=\frac{\pi}{3\alpha^2}.
\]
For the initial condition $u^{\alpha}_0=u^{\alpha}(0,\cdot)$, direct calculations then show that
\begin{equation}\label{norms}
\norm{u^{\alpha}_0}_{L^2}^2=\frac{4\pi\alpha}{3}, \quad
\norm{(u^{\alpha}_0)_x}_{L^2}^2=\frac{8\pi\alpha^3}{3\sqrt{3}}, \quad
\norm{(u^{\alpha}_0)_{xx}}_{L^2}^2=\frac{8\pi\alpha^5}{\sqrt{3}}.
\end{equation}
Upon inspection of the integrals involved, one easily sees that for all $k\in\N$, 
there is a constant $C_k>0$, independent of $\alpha$, such that
\begin{equation}\label{sobolev_norms}
\norm{\frac{\dif^ku^{\alpha}_0}{\dif x^k}}_{L^2}^2 = C_k \alpha^{2k+1}.
\end{equation}
For $\alpha\in(0,1)$, this completes the proof.
\end{proof}

\begin{remark}\label{rem}
\rm
(a) If $\alpha=\beta=\sqrt{\omega}/2$, then $u^{\alpha,\beta}(t,x)$ reduces to the usual soliton
$\e^{i\omega t}\ffi_\omega(x)$, with $\ffi_\omega$ defined in \eqref{sech}.

(b) A direct verification shows that the solution $u^{\alpha,\beta}(t,x)$ only blows up at $x=0$, i.e.~the 
denominator in \eqref{2param} never vanishes if $x\neq0$.

(c) The particular choice $\beta=\alpha/2$ enables one to compute explicitly the norms in \eqref{norms}. 
In fact, the relations \eqref{sobolev_norms} are easily derived by choosing $\beta=\gamma\alpha$ with, 
say, $\gamma\in(0,1)$, and using the change of variables $y=\alpha x$ in the integrals.
\end{remark}

\section{Instability of the solitons}\label{blow}

The blow-up instability of the solitons \eqref{sech} is now a consequence of Remark~\ref{rem}~(a). More precisely, fixing
$\alpha=\sqrt{\omega}/2$ and letting $\beta=\sqrt{\omega+\delta}/2$ with $0<\delta\ll1$, we obtain finite time
blow-up solutions $u^{\alpha,\beta}$ as close as we want to $\e^{i\omega t}\ffi_\omega(x)$.

\begin{theorem}
Fix $\omega>0$. For any $\ep>0$ there exists $q_{\omega,\ep}\in H^1(\R)$ such that 
\[
\norm{\ffi_\omega-q_{\omega,\ep}}_{H^1(\R)}<\ep 
\]
and the solution with initial data 
$u(0,\cdot)=q_{\omega,\ep}$ blows up in finite time.
\end{theorem}

\begin{proof}
Define $q_{\omega,\delta}(x)$ as $u^{\alpha,\beta}(0,x)$, with $\alpha=\sqrt{\omega}/2$ and 
$\beta=\sqrt{\omega+\delta}/2, \ \delta>0$, namely 
\[
q_{\omega,\delta}(x)=\frac{\sqrt{2}(\sqrt{\omega}+\sqrt{\omega+\delta})}{\e^{\sqrt{w}x}+\e^{-\sqrt{w+\delta}x}}.
\]
We only need to check that
\begin{equation}
\norm{\ffi_\omega-q_{\omega,\delta}}_{H^1} \to 0 \quad\text{as} \ \delta\to0.
\end{equation}
To show that
\begin{equation}\label{L2}
\int_\R |\ffi_\omega(x) - q_{\omega,\delta}(x)|^2\diff x 
\to 0 \quad\text{as} \ \delta\to0,
\end{equation}
we first observe that $|\ffi_\omega(x) - q_{\omega,\delta}(x)|\to 0$ as $\delta\to0$ for all $x\in\R$.
Furthermore if $0<\delta<1$, we have, for $-\infty<x\le0$,
\[
\frac{\sqrt{2}(\sqrt{\omega}+\sqrt{\omega+\delta})}{\e^{\sqrt{w}x}+\e^{-\sqrt{w+\delta}x}}
\le \frac{\sqrt{2}(\sqrt{\omega}+\sqrt{\omega+1})}{\e^{\sqrt{w}x}+\e^{-\sqrt{w}x}},
\]
while, for $0<x<\infty$,
\[
\frac{\sqrt{2}(\sqrt{\omega}+\sqrt{\omega+\delta})}{\e^{\sqrt{w}x}+\e^{-\sqrt{w+\delta}x}}
\le \frac{\sqrt{2}(\sqrt{\omega}+\sqrt{\omega+1})}{\e^{\sqrt{w}x}+\e^{-\sqrt{w+1}x}},
\]
and so \eqref{L2} follows by dominated convergence. Applying similar estimates to 
the derivative
\[
(q_{\omega,\delta})_x(x)=
\sqrt{2}(\sqrt{\omega}+\sqrt{\omega+\delta})
\frac{\sqrt{w+\delta}\,\e^{-\sqrt{w+\delta}x}-\sqrt{w}\,\e^{\sqrt{w}x}}{(\e^{\sqrt{w}x}+\e^{-\sqrt{w+\delta}x})^2}
\]
and using again dominated convergence, we also have
\begin{equation*}
\int_\R |(\ffi_\omega)_x(x) - (q_{\omega,\delta})_x(x)|^2\diff x 
\to 0 \quad\text{as} \ \delta\to0,
\end{equation*}
from which the conclusion follows.
\end{proof}

\section{Remarks on the defocusing case}\label{defoc}

The `defocusing' equation \eqref{defoc_nls}
has also been considered in \cite{mus1,mus2}. 
Our local well-posedness result, Theorem~\ref{local.thm}, 
carries over to \eqref{defoc_nls}, with an identical proof. 
On the other hand, it is shown in \cite[p.~936]{mus2} that `one-soliton' solutions of the type \eqref{2param} 
are not available in the defocusing case.
Global well-posedness for \eqref{defoc_nls} seems to be an open problem.



\end{document}